\documentclass[12pt]{article}
\usepackage{authblk}
\usepackage{amsmath}
\usepackage{amssymb}
\usepackage{amsfonts}
\usepackage{tikz}

\usepackage{amsthm}

\newtheorem*{notation*}{Notation}

\newtheorem{theorem}{Theorem}[section]
\newtheorem{proposition}[theorem]{Proposition}
\newtheorem{lemma}[theorem]{Lemma}
\newtheorem{claim}[theorem]{Claim}
\newtheorem{observation}[theorem]{Observation}

\newtheorem{conjecture}[theorem]{Conjecture}

\newtheorem*{remark*}{Remark}
\numberwithin{equation}{section}
\DeclareMathOperator{\sat}{sat}
\DeclareMathOperator{\wsat}{w-sat}
\DeclareMathOperator{\ssat}{s-sat}

\theoremstyle{definition}
\newtheorem*{definition}{Definition}

\def \F {\mathcal{F}}

\def \d {\delta}

\tikzstyle{cir} = [draw, circle, minimum height= 20 mm]

\title{Minimizing the number of edges in $K_{s,t}$-saturated bipartite graphs}
\author[1]{Debsoumya Chakraborti\thanks{This work was supported by the Institute for Basic Science (IBS-R029-C1)}}
\author[2]{Da Qi Chen \thanks{This material is based upon work supported by the Air Force Office of Scientific Research under award number FA9550-20-1-0080}}
\author[2]{Mihir Hasabnis}
\affil[1]{\small Discrete Mathematics Group, Institute for Basic Science (IBS), Daejeon,~South~Korea}
\affil[2]{\small Department of Mathematical Sciences, Carnegie Mellon University, Pittsburgh,~USA}
\affil[ ]{\small Email:
\texttt{debsoumya@ibs.re.kr},
\texttt{daqic@andrew.cmu.edu},
\texttt{mhasabni@andrew.cmu.edu} }
\begin{document}
\maketitle
\begin{abstract}
This paper considers an edge minimization problem in saturated bipartite graphs. An $n$ by $n$ bipartite graph $G$ is $H$-saturated if $G$ does not contain a subgraph isomorphic to $H$ but adding any missing edge to $G$ creates a copy of $H$. More than half a century ago, Wessel and Bollob\'as independently solved the problem of minimizing the number of edges in $K_{(s,t)}$-saturated graphs, where $K_{(s,t)}$ is the `ordered' complete bipartite graph with $s$ vertices from the first color class and $t$ from the second. However, the very natural `unordered' analogue of this problem was considered only half a decade ago by Moshkovitz and Shapira. When $s=t$, it can be easily checked that the unordered variant is exactly the same as the ordered case. Later, Gan, Kor\'andi, and Sudakov gave an asymptotically tight bound on the minimum number of edges in $K_{s,t}$-saturated $n$ by $n$ bipartite graphs, which is only smaller than the conjecture of Moshkovitz and Shapira by an additive constant. In this paper, we confirm their conjecture for $s=t-1$ with the classification of the extremal graphs. We also improve the estimates of Gan, Kor\'andi, and Sudakov for general $s$ and $t$, and for all sufficiently large $n$.      
\end{abstract}

\section{Introduction}

We start with a couple of standard definitions in the literature of graph saturation. For graphs $G$ and $H$, $G$ is said to be $H$-saturated if it does not contain a copy of $H$, but adding any missing edge to $G$ creates a copy of $H$. The saturation number $\sat(n, H)$ is defined as the minimum number of edges in an $n$-vertex $H$-saturated graph. Note that finding the saturation number is, in some sense, the dual of the classical Tur\'an problem.

Zykov \cite{Z}, and Erd\H{o}s, Hajnal, and Moon \cite{EHM} initiated the investigation of graph saturation by studying the saturation number for complete graphs. They proved that the number of edges in an $n$-vertex $K_s$-saturated graph is uniquely minimized by the graph obtained from joining all the edges between a $K_{s-2}$ and an independent set of order $n-s+2$. A few years later, Bollob\'as introduced the closely relation notion of weak saturation in \cite{B68}. For graphs $G$ and $H$, $G$ is called weakly $H$-saturated if it is possible to add all the missing edges to $G$ in some order such that each addition creates a new copy of $H$. Similar to the definition of the saturation number, the function $\wsat(n,H)$ is defined to be the minimum number of edges in an $n$-vertex weakly $H$-saturated graph. It follows easily from the definition that $\wsat(n,H) \le \sat(n,H)$ for any graph $H$. Interestingly, when $H$ is a clique, it is known that these two functions are equal (see, e.g., \cite{L}) even though their sets of extremal graphs are different. For example, if $H$ is just a triangle $K_3$, then all trees are weakly $K_3$-saturated, but only stars are $K_3$-saturated.

Erd\H{o}s, Hajnal, and Moon \cite{EHM} also introduced the notion of bipartite saturation. In particular, when $G$ and $H$ are both bipartite graphs, $G$ is $H$-saturated if $G$ does not contain $H$ but adding any edge across the bipartition creates a copy of $H$. Then, a natural objective is to find $\sat(K_{n,n},H)$, the minimum number of edges in an $H$-saturated $n$ by $n$ bipartite graph. They conjectured that $\sat(K_{n,n},K_{s,s}) = n^2 - (n-s+1)^2$, which is tight for the graph created by choosing $s-1$ vertices from both sides of the bipartite graph and connecting them to every vertex in the opposite side. A couple of years later, Wessel \cite{W} and Bollob\'as \cite{B67} independently confirmed the conjecture in a more general `ordered' setting. In particular, this variant imposes an extra condition by ordering the two vertex classes in $G$ and $H$ and requires that adding any missing edge creates a copy of $H$ respecting the order, i.e., the first class of $H$ should be in the first class of $G$. For example, let $K_{(s,t)}$ denote the complete `ordered' bipartite graph with $s$ vertices in the first class and $t$ vertices in the second class. Then, $G$ is called $K_{(s,t)}$-saturated if $G$ is $K_{(s,t)}$-free (but may contain a copy of $K_{(t, s)}$) and adding any missing edge across the bipartition creates a $K_{(s,t)}$. Wessel and Bollob\'as proved that $\sat(K_{n,n},K_{(s,t)}) = n^2 - (n-s+1)(n-t+1)$. Later, Alon \cite{A} generalized this to $k$-uniform hypergraphs in a $k$-partite setting. Furthermore, they showed that the ordered saturation number is the same as the ordered weak saturation number. We refer the readers to the informative survey \cite{FFS} to see other classical results in graph saturation.

As for the more natural unordered setting for bipartite saturation, this half-century-old problem only started receiving attention in the recent years (see, e.g., \cite{BBMR}, \cite{MS}). Moshkovitz and Shapira \cite{MS} showed that the unordered weak saturation number, $\wsat(K_{n,n},K_{s,t})$, is $(2s-2+o(1))n$ with $s\le t$ which is surprisingly significantly smaller (roughly by $(t-s)n$) than the corresponding ordered weak saturation number. One might think that there might exist a similar gap between the unordered and ordered saturation numbers as well. However, Moshkovitz and Shapira \cite{MS} conjectured that there is at most a constant additive gap:

\begin{conjecture}[\cite{MS}] \label{conjecture}
Let $1 \le s < t$ be fixed. Then, there exists an $N = N(s,t)$ such that for all $n \ge N$, we have the following: $$\sat(K_{n,n},K_{s,t}) = (s+t-2)n - \left\lfloor \left(\frac{s+t-2}{2}\right)^2 \right\rfloor.$$ 
\end{conjecture}

In this setting, we often assume that $s < t$. Note that as previously mentioned, when $s=t$, the unordered saturation number is already known and is equal to the ordered saturation number.

Moshkovitz and Shapira \cite{MS} have also constructed the following class of $K_{s,t}$-saturated graphs that matches the number of edges given in the above conjecture. First choose $s-1$ vertices from both color classes and connect them to all the vertices from the other side. Then, pick $\ell = \left \lfloor \frac{t-s}{2} \right \rfloor$ additional vertices from each class and add all edges between them to form a $K_{\ell, \ell}$. Lastly, add edges amongst the remaining vertices such that they all end up with degree exactly $t-1$ while ensuring that $G$ does not contain a copy of $K_{s,t}$. Observe that any missing edge $e$ in $G$ is incident to a vertex $u$ with degree exactly $t-1$. Then, adding $e$ creates a copy of $K_{s,t}$ by using the endpoints of $e$ along with the $t-1$ neighbors of $u$ and the $s-1$ vertices of degree $n$ on the same side as $u$. With a simple counting, one can see that these graphs have $(s+t-2)n - \left\lfloor \left(\frac{s+t-2}{2}\right)^2 \right\rfloor$ edges. Note that when $s< t-1$, these graphs contain less edges than  $\sat\left(K_{n, n}, K_{(s,t)}\right)$, showing that $\sat\left(K_{n, n}, K_{s,t}\right) < \sat\left(K_{n, n}, K_{(s,t)}\right)$.

Gan, Kor\'andi, and Sudakov \cite{GKS} proved that $\sat(K_{n,n},K_{s,t}) \ge (s+t-2)n - (s+t-2)^2$, which is only smaller than the above conjecture by an additive constant. In addition, they gave another family of extremal graphs achieving the bound in Conjecture \ref{conjecture}; we encourage our readers to consult the third section of \cite{GKS} for an exact description of these extremal graphs. They also noted that the existence of such variety of examples (most of which are asymmetric) provides further evidence in the difficulties of this conjecture.

As mentioned in \cite{MS}, it is an easy exercise to check that Conjecture \ref{conjecture} is true when $s=1$. Another interesting case to investigate is the other extreme value of $s$, i.e., when $s=t-1$. In this context, the specific instance of $\sat\left( K_{n, n}, K_{2, 3}\right)$ was already determined by Gan, Kor\'andi, and Sudakov \cite{GKS}. We resolve this case by proving that Conjecture \ref{conjecture} is correct for any $s=t-1$, and providing the classification of all possible extremal graphs. In order to describe the extremal structures, we define $\F^n_{s,t}$ to be the collection of $K_{s,t}$-free $n$ by $n$ bipartite graphs where one of the color classes contains a set $S$ of $s-1$ vertices such that $S$ is connected to all vertices on the other side and all other vertices on the same side as $S$ have degree exactly $t-1$. It is a simple exercise to check that these graphs are $K_{s,t}$ saturated (in fact, each graph is also ordered saturated) and they contain $(s+t-2)n-(s-1)(t-1)$ number of edges. Then, we prove the following:

\begin{theorem} \label{main1}
Let $t > 1$ be fixed. Then, there exists an $N = N(t)$ such that for all $n \ge N$, we have the following: $$\sat(K_{n,n},K_{t-1,t}) = (2t-3)n - (t-1)(t-2).$$
Moreover, a graph $G$ is an extremal graph achieving the equality if and only if $G \in \F^n_{t-1,t}$.
\end{theorem}

We also improve the lower bound of Gan, Kor\'andi, and Sudakov \cite{GKS} for general $s$ and $t$.

\begin{theorem} \label{main2}
Let $1 \le s < t$ be fixed. Then, there exists an $N = N(s,t)$ such that for all $n \ge N$, we have the following: $$\sat(K_{n,n},K_{s,t}) \ge (s+t-2)n - (t-1)(t-2) - \left\lfloor \frac{(s-1)^2}{4}\right\rfloor.$$
\end{theorem}

This paper is organized in the following way. In Section \ref{sec:prelim}, we start with the easy case of proving Theorems \ref{main1} and \ref{main2} with the assumption that the minimum degree is less than $t-1$. This will also showcase some proof strategies and contain certain definitions used throughout the paper. Section \ref{sec:asymp} contains the heart of the paper, where we deduce that any extremal graph with minimum degree at least $t-1$ contains a `nice' substructure (formally defined in Section \ref{sec:prelim}). In Section \ref{sec:core}, we analyze the obtained nice substructure and finish the proof of Theorems \ref{main1} and \ref{main2}. Finally, in Section \ref{sec:smallcase}, we verify Conjecture \ref{conjecture} for the next open case of $K_{2,4}$-saturation. Lastly, we end with a few concluding remarks. 

\section{Preliminaries}\label{sec:prelim}

In this section, our main goal is to prove a relatively easier instance of Theorems \ref{main1} and \ref{main2}, where we make the additional assumption that the graph has minimum degree less than $t-1$. These short proofs also highlight certain techniques that can be used for the general case. They also motivate us to define and search for certain nice structures (cores) in the general graph which is the strategy we employ in the later sections. Thus, let us assume that $1 \le s < t$, $n$ is sufficiently large, and $G$ is a $K_{s,t}$-saturated bipartite graph with vertex classes $U$ and $U'$ of size $n$. This means that $G$ does not contain a copy of $K_{s,t}$, but adding any missing edge between $U$ and $U'$ creates a copy of $K_{(s,t)}$ or $K_{(t,s)}$. Here $K_{(a,b)}$ refers to a complete bipartite graph with $a$ vertices in $U$ and $b$ vertices in $U'$. 

We start with proving Theorem \ref{main2} in the scenario when $G$ has minimum degree less than $t-1$.  The result actually follows immediately from the argument used to prove Proposition 2.1 in \cite{GKS}. We repeat the same proof for the sake of completeness and to demonstrate certain proof strategies that are used later throughout the paper. 
 
\begin{proposition} \label{mindeg}
Suppose a $K_{s,t}$-saturated graph $G$ has minimum degree $\d < t-1$. Then $G$ contains at least $(s+t-2)n - (t-1)(t-2)$ edges.  
\end{proposition}

\begin{proof}
Since adding any edge creates a copy of $K_{s,t}$, the minimum degree of $G$ is at least $s-1$, implying $s-1 \le \d < t-1$. Let $u_0$ be a vertex of degree $\d$ and $N(u_0)$ be its neighborhood. Without loss of generality, we can assume that $u_0 \in U$. Since $u_0$ has less than $t-1$ neighbors, for any vertex $u' \in U' \setminus N(u_0)$, adding an edge $u_0u'$ to $G$ creates a copy of $K_{(t,s)}$ instead of $K_{(s, t)}$. Thus given such vertex $u'\in U'\setminus N(u_0)$, let $S_{u'}$ be the set of $t-1$ vertices in $U\setminus u_0$ which is in a created copy of $K_{(t,s)}$. Note that the vertices in $S_{u'}$ also  have $s-1$ common neighbors with $u_0$. Define $V \subseteq U$ to be the union of these $S_{u'}$'s. Then all the vertices in $V$ are adjacent to at least $s-1$ neighbors in $N(u_0)$, and all the vertices in $U' \setminus N(u_0)$ are adjacent to at least $t-1$ vertices in $V$. Now, we count the number of edges in $G$ in the following manner. For two disjoint sets of vertices $A$ and $B$, we write $e(A, B)$ to denote the number of edges between $A$ and $B$.

\begin{align}
    e(U, U') &= e(V, N(u_0)) + e(V, U' \setminus N(u_0)) + e(U \setminus V, U') \nonumber\\
    &\ge (s-1)|V| + (t-1)(n-\d) + \d(n-|V|) \label{match}\\
    &\ge (s-1)|V| + (t-1)(n-t+2) + (s-1)(n-|V|) \nonumber\\
    &= (s+t-2)n - (t-1)(t-2), \nonumber
\end{align}
proving the proposition. 
\end{proof}

Note that the proof relied on finding a very large set ($U'\setminus N(u_0)$) with the property that the large set contains vertices of degree at least $t-1$ while each vertex on the other side can still find at least $s-1$ edges somewhat avoiding this large set. In later sections, we generalize this idea by finding two large sets (analogous to $V$ and $U'\setminus N(u_0)$) with the property that one large set (e.g. $U'\setminus N(u_0)$) has degree at least $t-1$ and the second large set has at least $s-1$ neighbors outside of the first large set (e.g. between $V$ and $N(u_0)$). This will asymptotically make up for most of the edges and thus a common strategy throughout the paper is to identify these two large sets and their complementary small sets. Note that any additional edges between the two small sets will help with the lower bound, thus we are motivated to make them as small and dense as possible. In addition, the above proof found the large sets $U'\setminus N(u_0)$ and $V$ with the desirable properties by adding edges between a well-chosen vertex $u_0$ and vertices in the large set on the opposite side. This works only if the edge is not in $G$ to begin with, suggesting that there should be a vertex ($u_0$) whose entire neighbourhood is inside the small set on the other side. This motivates the following definition of a \textit{core} which was first introduced in \cite{GKS}. 

\begin{definition}
    A \textit{core} is a subgraph induced by a vertex-set $A\cup A'$ with $A\subseteq U$, $A'\subseteq U'$, and there exist $a_0\in A$ and $a_0'\in A'$ such that their neighborhoods are also contained in $A\cup A'$.
\end{definition}

In \cite{GKS}, a core of $G$ was found by a simple averaging argument. In our paper, majority of the effort is spent on finding a `nice' core that is small and dense. This is enough to improve the lower bound provided by the structural arguments in \cite{GKS}. To achieve this, we need $n$ to be sufficiently large in terms of $s$ and $t$. Our strategy is to look for the following nice core in $G$:

\begin{definition}
    A \textit{nice core} is a core containing a copy of $K_{s,t-1}$ satisfying the following properties: there exist $a_0\in A$ and $a_0'\in A'$ with $N(a_0)=A'$, $N(a_0')=A$, and $|A|=|A'|=t-1$. 
\end{definition}

 In Section \ref{sec:asymp}, we argue that if $G$ has minimum degree at least $t-1$ and does not have a nice core, then $G$ already has at least $(s+t-2)n - (s-1)(t-1)$ edges. In the subsequent section, we deal with the case when $G$ contains a nice core by using similar structural arguments as those in \cite{GKS}.

Now, we further analyze the equations in the proof of Proposition \ref{mindeg} in order to classify the extremal structures when $G$ has minimum degree less than $t-1$ and $s=t-1$.

\begin{proposition}
     A $K_{t-1, t}$-saturated graph $G$ with minimum degree less than $t-1$ has minimum number of edges if and only if $G$ is a graph in $\F_{t-1,t}^n$ with minimum degree $t-2$. 
\end{proposition}

\begin{proof}

    Note that any graph in $\F_{t-1,t}^n$ contains exactly $(2t-3)n-(t-1)(t-2)$ many edges. Thus, by Proposition \ref{mindeg}, it remains to show that any $K_{t-1, t}$-saturated graph with that many edges must be one of the graphs in $\F_{t-1,t}^n$.
    
    Suppose for the sake of contradiction that $G$ is a $K_{t-1,t}$-saturated graph with exactly $(2t-3)n - (t-1)(t-2)$ edges but $G$ is not one of the graphs in $\F_{t-1,t}^n$. Note that the minimum degree $\d = t-2$ and $|N(u_0)|=t-2$. By construction, every vertex in $V$ is adjacent to all vertices in $N(u_0)$. Due to the equality in Step \eqref{match}, we know that every vertex in $U' \setminus N(u_0)$ has exactly $t-1$ neighbors in $V$, and every vertex in $U \setminus V$ has degree exactly $t-2$. If all vertices in $U$ are adjacent to the $t-2$ vertices in $N(u_0)$, then $G$ is indeed a graph in $\F_{t-1,t}^n$, a contradiction.

    Thus, we may assume there exists a vertex $u \in U$ which is not adjacent to all the vertices in $N(u_0)$. Thus, $u \not \in V$ and has degree exactly $t-2$. Let $u' \in U' \setminus N(u_0)$ be a vertex that is also not adjacent to $u$ and consider adding the edge $uu'$ to $G$. It must create a copy of $K_{(t,t-1)}$ because $u$ has only $t-2$ neighbors in $U'$. Let $S$ be the rest of the $t-1$ vertices on the same side as $u$ in the created copy of $K_{(t,t-1)}$ other than $u$. Note that every vertex in $S$ has at least $t-1$ neighbors. Since vertices outside of $V$ have degree $t-2$, we have $S\subseteq V$. Since $u$ has less than $t-2$ neighbors in $N(u_0)$, there exists at least one vertex $v'$ other than $u'$ in $U'\setminus N(u_0)$ that is also in the created copy of $K_{(t,t-1)}$. Then, the vertices in $S$ along with the vertices in $\{u', v'\} \cup N(u_0)$ form a copy of $K_{(t-1,t)}$ in $G$, a contradiction.   
\end{proof}

From now on, we may assume that $G$ has minimum degree at least $t-1$. 

\section{Finding a nice core}
\label{sec:asymp}

Let $G$ be a $K_{s,t}$-saturated $n$ by $n$ bipartite graph with at most $(s+t-2)n - (s-1)(t-1)$ many edges and minimum degree at least $t-1$. Our goal in this section is to show that $G$ contains a copy of the nice core, i.e., there exist $A\subseteq U$ and $A'\subseteq U'$ such that $|A|=|A'|=t-1$, there exist $a_0\in A$ and $a_0'\in A'$ such that $N(a_0)=A'$ and $N(a_0')=A$, and the graph induced by $(A, A')$ contains a copy of $K_{s,t-1}$. More precisely, we will prove the following lemma:

\begin{lemma}
\label{findcore}
If $G$ does not contain a nice core, then $G$ is one of the graphs in $\F_{s,t}^n$. 
\end{lemma}

The general strategy is to divide the vertex classes of $G$ into sets that either have large degree or have large size (this idea of dividing the vertex set has been used previously for different saturation problems in \cite{BFP} and \cite{CL}). Then, we slowly refine the sets to obtain more structural properties to show that the graph $G$ contains either too many edges or a nice core.

Let $V_0 := \{v \in U: d(v) \ge n^{\frac{1}{4}}\}$ and similarly define $V_0' := \{v' \in U': d(v') \ge n^{\frac{1}{4}}\}$. Due to the upper-bound on the number of edges of $G$, it follows that the sizes of $V_0$ and $V_0'$ are at most $O\left(n^{\frac{3}{4}}\right)$. We will slowly refine these sets through this section (see Figure \ref{fig:asymp}). 

\begin{figure}
    \centering
    \includegraphics[scale=0.7]{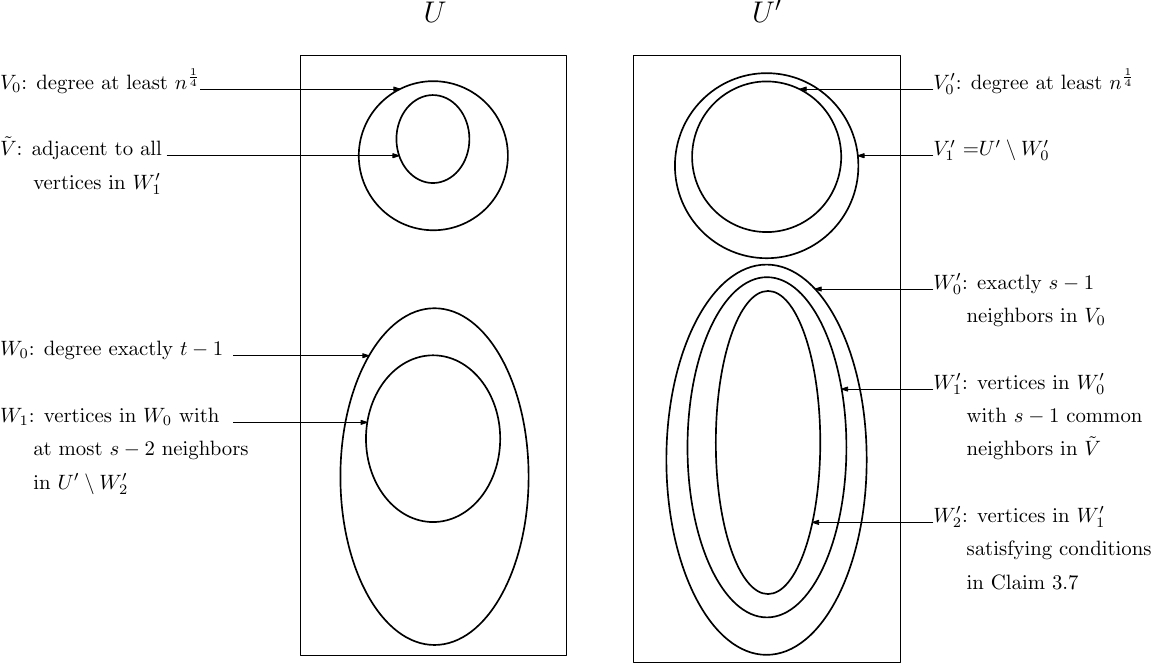}
    \caption{Structure obtained in Section \ref{sec:asymp}}
    \label{fig:asymp}
\end{figure}

First, we provide three small technical observations that we will apply repeatedly in the future. 

\begin{observation}
\label{bigsmall}
If there exist sets $R_{big}$ and $R_{small}$ on opposite sides of the bipartition in $G$ such that $|R_{big}|=n-o(n)$ and $|R_{small}|=o(n)$, and every vertex $u\in R_{big}$ has at least $s-1$ neighbors in $R_{small}$, then there exist $n-o(n)$ vertices in $R_{big}$ with exactly $s-1$ neighbors in $R_{small}$, and there are $n-o(n)$ many vertices outside but on the same side as $R_{small}$ with degree exactly $t-1$. 
\end{observation}

\begin{proof}
    Without loss of generality, assume that $R_{big}\subseteq U$ and $R_{small}\subseteq U'$. Between $R_{big}$ and $R_{small}$, there are at least $(s-1)(n-o(n))$ many edges. Since $G$ has minimum degree at least $t-1$, there are at least $(t-1)(n-o(n))$ edges between $U$ and $U'\setminus R_{small}$. Since $G$ has at most $n(s+t-2)-o(n)$ many edges, there are $n-o(n)$ vertices in $R_{big}$ with exactly $s-1$ neighbors in $R_{small}$ and there are $n-o(n)$ many vertices in $U'\setminus R_{small}$ with degree exactly $t-1$, proving our observation. 
\end{proof}

The following two observations are often used in conjunction. 

\begin{observation}
\label{threepath}
Let $R\subseteq U$ and $R'\subseteq U'$. Let $u\in R$ and $u'\in R'$ such that $u$ and $u'$ are at distance more than three in the graph induced by $(R, R')$ (i.e., there is no path using at most three edges between $u$ and $u'$ in the graph induced by $(R, R')$). Then, for one of the sets $R$ or $R'$, the copy of $K_{s,t}$ created by adding the edge $uu'$ uses only one vertex from that set (namely, either $u$ or $u'$). 
\end{observation}

\begin{proof}
    Suppose for the sake of contradiction that both $R$ and $R'$ contain another vertex $x$ and $x'$ respectively that are also in the copy of $K_{s,t}$ created by adding the edge $uu'$. Then, $ux'xu'$ is a path in $(R,R')$, contradicting the fact that $u'$ is at distance more than three away from $u$.  
\end{proof}

\begin{observation}
\label{threepathlarge}
    Let $R\subseteq U\setminus V_0$ and $R'\subseteq U'\setminus V_0'$, where $|R'|=n-o(n)$. Then, for every vertex $u\in R$, there exists $R_0'\subseteq R'$ with $|R_0'|=n-o(n)$ such that every vertex in $R_0'$ is at distance more than three away from $u$ in the graph induced by $(R, R')$. 
\end{observation}

\begin{proof}
    Since all the vertices outside of $V_0$ and $V_0'$ have degree at most $n^{\frac{1}{4}}$, it follows that there are at most $3n^{\frac{3}{4}}$ vertices in $R'$ at distance three or less from $u \in R$. Thus, $|R_0'|$ is $n-o(n)$.
\end{proof}

We will now prove Lemma \ref{findcore} through a series of claims. In our claims, we will often obtain large sets $R$ and $R'$ with certain structural properties. Then, we will pick a vertex $u\in R$ and apply Observations \ref{threepath} and \ref{threepathlarge} to obtain a refined large set $R_0'\subseteq R'$ with even more structural properties. 

Now, let us begin the series of claims to find the nice core. 

\begin{claim} \label{c1}
One of the following is true: either
\begin{enumerate}
    \item there exist $n - o(n)$ vertices of $U' \setminus V_0'$ each having exactly $s-1$ neighbors in $V_0$ and $n - o(n)$ vertices of $U \setminus V_0$ each having degree exactly $t-1$, or
    \item the mirrored opposite where there exist $n - o(n)$ vertices of $U \setminus V_0$ each having exactly $s-1$ neighbors in $V_0'$ and $n - o(n)$ vertices of $U' \setminus V_0'$ have degree exactly $t-1$.
\end{enumerate}
\end{claim}

\begin{proof}
    First, suppose that every vertex in $U \setminus V_0$ has at least $s-1$ neighbors in $V_0'$. Then, setting $R_{big}=U \setminus V_0$ and $R_{small}=V_0'$, it follows from Observation \ref{bigsmall} that the second situation occurs.

    Now, suppose that there exists a vertex $u\in U \setminus V_0$ with less than $s-1$ neighbors in $V_0'$. Let $W'\subseteq U' \setminus V_0'$ contain all the vertices $w'$ such that $w'$ is at distance more than three from $u$ in the graph induced by $\left(U \setminus V_0, U' \setminus V_0'\right)$. It follows from Observation \ref{threepathlarge} that $|W'|=n-o(n)$. Let $w' \in W'$ and consider a copy of $K_{s,t}$ created by adding the edge $uw'$. Since $u$ has at most $s-2$ neighbors in $V_0'$, the copy of $K_{s,t}$ uses at least one more vertex other than $w'$ in $U' \setminus V_0'$. Taking $R$ and $R'$ as $U \setminus V_0$ and $U' \setminus V_0'$ respectively, it follows from Observation \ref{threepath} that this $K_{s,t}$ only uses $u$ in $U \setminus V_0$ and takes at least $s-1$ other vertices from $V_0$. Then, $w'$ has at least $s-1$ neighbors in $V_0$. Thus setting $R_{big}=W'$ and $R_{small}=V_0$, applying Observation \ref{bigsmall} ensures the occurrence of the first situation. 
\end{proof}

From now on, without loss of generality, we assume that $G$ satisfies property 1 in Claim \ref{c1}. Let $W_0$ be the large set of vertices in $U\setminus V_0$ with degree exactly $t-1$. Let $W_0'\subseteq U'\setminus V_0'$ be the large set of vertices with exactly $s-1$ neighbors in $V_0$.  For convenience, define $V_1'=U'\setminus W_0'$. 

\begin{claim} \label{c2}
If $G$ does not contain a nice core, then there exists a set $W_1'\subseteq W_0'$ containing $n-o(n)$ vertices such that they have $s-1$ common neighbors in $V_0$.  
\end{claim}

\begin{proof}
    \textbf{Case 1}: There exists a vertex $w\in W_0$ that has at most $s-2$ neighbors in $V_1'$. Let $W_1'\subseteq W_0'$ be the set of vertices that are at distance more than three away from $w$ in the graph induced by $(U\setminus V_0, W_0')$. It follows from Observation \ref{threepathlarge} that $|W_1'|=n-o(n)$. Let $x'\in W_0'$ be a neighbor of $w$ (such a vertex exists because $w$ has only $s-2$ neighbors in $V_1'$ but it has degree $t-1$). Note that $x'$ has exactly $s-1$ neighbors in $V_0$. We will show that every vertex in $W_1'$ also has the same neighbors as $x'$ in $V_0$, resulting in the desirable structure.
    
    Take any $w'\in W_1'$ and consider a copy of $K_{s,t}$ created by adding the edge $ww'$. Since $w$ has at most $s-2$ neighbors in $V_1'$ and $w'$ has exactly $s-1$ neighbors in $V_0$, creating a copy of $K_{(t, s)}$ involves at least two vertices from $U\setminus V_0$ and two vertices from $W_0'$, contradicting Observation \ref{threepath}. Thus, a copy of $K_{(s, t)}$ is created. Since $w\in W_0$ has degree exactly $t-1$, $x'$ is used in the $K_{(s, t)}$. Once again, by Observation \ref{threepath}, this copy of $K_{(s, t)}$ cannot use any more vertices in $U\setminus V_0$ and thus must use all $s-1$  neighbors of $x'$ in $V_0$, proving $w'$ shares the same neighbors with $x'$ in $V_0$.
    \smallskip
    
    \textbf{Case 2}: Every vertex in $W_0$ has at least $s-1$ neighbors in $V_1'$. Setting $R_{big}=W_0$ and $R_{small}=V_1'$, applying Observation \ref{bigsmall} provides us a large set $X'\subseteq W_0'$ with degree exactly $t-1$. Another application of Observation \ref{bigsmall} with $R_{big} = X'$ and $R_{small} = U \setminus W_0$ provides us another large set $X'' \subseteq X'$ with exactly $s-1$ neighbors in $U \setminus W_0$. Our goal is to find an induced matching of size two within the sets $(W_0, X'')$ in order to find a nice core. Since the vertices in $W_0$ and $X''$ have degree $t-1$ and $|X''|$ is large, there exist two vertices $w', x'\in X''$ that do not share a common neighbor in $W_0$. Since the vertices in $X''$ have exactly $s-1$ neighbors in $U \setminus W_0$, both $w'$ and $x'$ have at least one neighbor in $W_0$, say $w$ and $x$ respectively. Note that $wx'$ and $xw'$ are not edges and all four vertices have degree exactly $t-1$. 
    
    Now consider adding the edge $wx'$. If it creates a copy of $K_{(s, t)}$, it uses all $t-1$ neighbors of $w$ and thus also uses $w'$. Then setting $A=N(w')$ and $A'=N(w)$ provides the desirable core. A similar argument can be made if a copy of $K_{(t, s)}$ was created, contradicting our assumption. 
\end{proof}

Following from the previous claim, we can refine the set $W_0'$ further into a large set $W_1'\subseteq W_0'$ such that every vertex $w'\in W_1'$ has the same exact $s-1$ neighbors in $V_0$. Let $\tilde{V} \subseteq V_0$ be that set of common neighbors. The next claim will refine the set $W_1'$ even further.

\begin{claim} \label{c3}
If $G$ does not contain a nice core, then there exists $ W_2'\subseteq W_1'$ where $|W_2'|=n-o(n)$ and $N(w')\subseteq \tilde{V} \cup W_0$ for every $w'\in W_2'$. Furthermore, if $W_1\subseteq W_0$ is defined as the set of vertices with at most $s-2$ neighbors in $U'\setminus W_2'$, then $W_1$ is non-empty and $|W_0 \setminus W_1|< |W_2'|-3n^{3/4}$. 
\end{claim}

This technical claim implies that the size of the set $W_0 \setminus W_1$ cannot be very close to $n$.

\begin{proof}
    Set $R_{big}=W_1'$ and $R_{small}=U \setminus W_0$, applying Observation \ref{bigsmall} provides a large set $W_2'\subseteq W_1'$ such that all vertices $w'\in W_2'$ have exactly $s-1$ neighbors in $U \setminus W_0$. Since $W_2'\subseteq W_1'$, every vertex in $W_2'$ already has $s-1$ neighbors in $\tilde{V}$, implying all its other neighbors are in $W_0$. Then, it remains to check the condition in the furthermore part of the claim. 
    
    Suppose for the sake of contradiction that $|W_0\setminus W_1|\ge \min \left(|W_0|, |W_2'|-3n^{3/4}\right) = n-o(n)$. We can use a similar argument as Case 2 of the previous claim to find the nice core. Consider $R_{big}= W_0\setminus W_1$ and $R_{small}=U'\setminus W_2'$. Applying Observation \ref{bigsmall} obtains a large set $X'\subseteq W_2'$  with degree exactly $t-1$. Once again, one can find two vertices $w', x'\in X'$ with no common neighbors in $W_0$. Then one can find two vertices $w, x\in W_0$ such that $ww'$ and $xx'$ are edges but $wx'$ and $xw'$ are not. Lastly, by adding the edge $wx'$, one can find a nice core by either taking $(N(w'), N(w))$ or $(N(x'), N(x))$, a contradiction. 
\end{proof}

The condition stated in the furthermore part of the previous claim might look rather arbitrary but it is actually motivated by the last portion of the proof in the next claim. At some point, we need to find a special vertex in $W_2'$, but we might have to first eliminate as many as $|W_0\setminus W_1|$ vertices and an additional $3n^{3/4}$ (due to a distance more than three requirement). The condition essentially ensures the existence of that vertex.

\begin{claim} \label{c4}
If $G$ does not contain a nice core, then all the vertices in $U'$ are adjacent to the vertices in $\tilde{V}$.
\end{claim}

\begin{proof}
    Let $w'\in U'\setminus W_2'$. We break into three cases depending on $w'$'s interaction with $W_0$ and $W_1$. 
    \smallskip
    
    \textbf{Case 1}: $N(w')\cap W_0=\emptyset$. Let $w\in W_1$. Consider a copy of $K_{s,t}$ created by adding $ww'$. Since $w\in W_1$ has at most $s-2$ neighbors outside of $W_2'$, the copy uses at least one vertex $x'\in W_2'$. Since $w'$ has no neighbors in $W_0$, the $K_{s,t}$ uses at least $s-1$ vertices from $U\setminus W_0$. Since $x'\in W_2'$, the only neighbors it has outside of $W_0$ are the $s-1$ vertices in $\tilde{V}$. It follows that $w'$ is also adjacent to all the vertices in the set $\tilde{V}$.
    \smallskip
    
    \textbf{Case 2}: $w'$ has a neighbor $w\in W_1$. Let $x'\in W_2'$ be a vertex at distance more than three away from $w$ in the graph induced by $(W_0, W_2')$. Consider the copy of $K_{s,t}$ created by adding the edge $wx'$. Since $w$ has at most $s-2$ neighbors outside of $W_2'$, the $K_{s,t}$ uses at least two vertices in $W_2'$. It follows from Observation \ref{threepath} that the $K_{s,t}$ must use at least $s-1$ vertices outside of $W_0$. However, the only neighbors of $x'$ outside of $W_0$ are the $s-1$ vertices in $\tilde{V}$. This implies only a copy of $K_{(s, t)}$ can be created and it uses all $s-1$ vertices in $\tilde{V}$ and all $t-1$ neighbors of $w$. Since $w'\in N(w)$, it follows that $w'$ is adjacent to all vertices in $\tilde{V}$. 
    \smallskip
    
    \textbf{Case 3}: $N(w')\cap W_1=\emptyset$ but $N(w')\cap W_0\neq \emptyset$. Let $w\in N(w')\cap W_0$. We claim that there exists a vertex $x'\in W_2'$ that has less than $t-s$ many neighbors in $W_0\setminus W_1$ and is at distance more than three away from $w$ in the graph induced by $(W_0, W_2')$.
    
    Let $X'\subseteq W_2'$ be the set of vertices with at least $t-s$ neighbors in $W_0\setminus W_1$. Note that the vertices in $W_0\setminus W_1$ have degree exactly $t-1$ and have at least $s-1$ neighbors outside of $W_2'$. Then, each vertex in $W_0\setminus W_1$ has at most $t-s$ neighbors in $X'$. Then, by the Handshake Lemma, it follows that $|X'|\le |W_0\setminus W_1|$. Since every vertex in $(W_0, W_2')$ has degree at most $n^{1/4}$, the set of vertices at distance three or less from $w$ is at most $3n^{3/4}$. Then, it follows from the furthermore condition in the previous claim that there exists a vertex $x'\in W_2'$ at distance more than three away from $w$ and has less than $t-s$ neighbors in $W_0\setminus W_1$.
    
    Now, consider the copy of $K_{s,t}$ created by adding the edge $wx'$. Suppose we created a copy of $K_{(s, t)}$. Since $w$ has degree $t-1$, all of its neighbors are used, including $w'$. Since every vertex in $W_0$ has degree exactly $t-1$, the $K_{(s, t)}$ must use $s-1$ other vertices from outside of $W_0$. However, the only neighbors of $x'$ outside of $W_0$ are the $s-1$ vertices in $\tilde{V}$, implying that $w'$ is also adjacent to $\tilde{V}$. 
    
    However, if a copy of $K_{(t, s)}$ is created, since $x'$ has $s-1$ neighbors outside of $W_0$ and less than $t-s$ neighbors in $W_0\setminus W_1$, the $K_{(t, s)}$ uses at least one more vertex $u\in W_1$. Since $u$ has at most $s-2$ neighbors outside $W_2'$, the $K_{(t, s)}$ uses one more vertex in $W_2'$ other than $x'$, contradicting Observation \ref{threepath}. 
\end{proof}

\begin{proof}[Proof of Lemma \ref{findcore}]
    Using Claim \ref{c4}, we can conclude that if $G$ does not contain a nice core, then $G$ has the following structure: there is a set $\tilde{V} \subseteq U$ of size $s-1$ such that $(\tilde{V}, U')$ induces a complete bipartite graph $K_{(s-1,n)}$. Moreover, every vertex in $U$ has degree at least $t-1$. We claim that every vertex in $U \setminus \tilde{V}$ has degree exactly $t-1$. Suppose not, then find $u \in U \setminus \tilde{V}$ with degree at least $t$. Pick a set $T' \subseteq U'$ of $t$ neighbors of $u$, then $\{u\} \cup \tilde{V}$ and $T'$ induce a copy of $K_{(s,t)}$, a contradiction. Thus, every vertex in $U\setminus \tilde{V}$ has degree exactly $t-1$, resulting in $G$ being a graph in $\F_{s,t}^n$, completing our proof. 
\end{proof}

\section{When $G$ contains a nice core}
\label{sec:core}

Let $G$ be a $K_{s,t}$-saturated $n$ by $n$ bipartite graph with minimum degree at least $t-1$ and with minimum possible number of edges. Furthermore, assume that $G$ contains a copy of the nice core (sets $A\subseteq U$ and $A'\subseteq U'$ of size $t-1$ containing vertices $a_0 \in A$ and $a_0' \in A'$ with $N(a_0)=A'$ and $N(a_0')=A$, and containing a copy of $K_{s,t-1}$). Our goal in this section is to show that $G$ has at least $(s+t-2)n - (t-1)(t-3) - s - \left\lfloor \frac{(s-1)^2}{4}\right\rfloor$ many edges. Moreover, when $s = t-1$, we show that $G$ has to be one of the graphs in $\F_{t-1,t}^n$. Our techniques in this section are similar to the ones used in \cite{GKS} with some modifications for the special case $s=t-1$.

\begin{lemma} \label{nicecore}
If $G$ contains a nice core, then $G$ contains at least $(s+t-2)n - (t-1)(t-3) -s - \left\lfloor \frac{(s-1)^2}{4}\right\rfloor$ many edges. Moreover, when $s=t-1$, then $G \in \F_{t-1,t}^n$.
\end{lemma}

    We partition $U \setminus A$ into two parts by defining $B$ to be the set of vertices in $U \setminus A$ having at least $s-1$ neighbors in $A'$, and $C$ to be the set $U \setminus (A \cup B)$. Similarly, define $B'$ and $C'$. Furthermore, break $B$ into two parts $B_1$ and $B_2$, by defining $B_1$ to be the set of vertices having at least $t-1$ neighbors in $A' \cup B'$. Similarly, define $B_1'$ and $B_2'$. Now, break $C$ into two parts $C_1$ and $C_2$, by setting $C_1$ to be the vertices in $C$ having at least $s-1$ neighbors outside of $B_2'$. Similarly, define $C_1'$ and $C_2'$.  Refer to Figure \ref{fig:core} for a graphic representation of these partitions. Keep in mind that the partition on the $U'$ side is defined symmetrically to the $U$ side. 
    
    We first provide two observations. These were also noted in \cite{GKS} but we will include their proofs for the sake of completeness.

\begin{figure}
    \centering
    \includegraphics[scale=0.8]{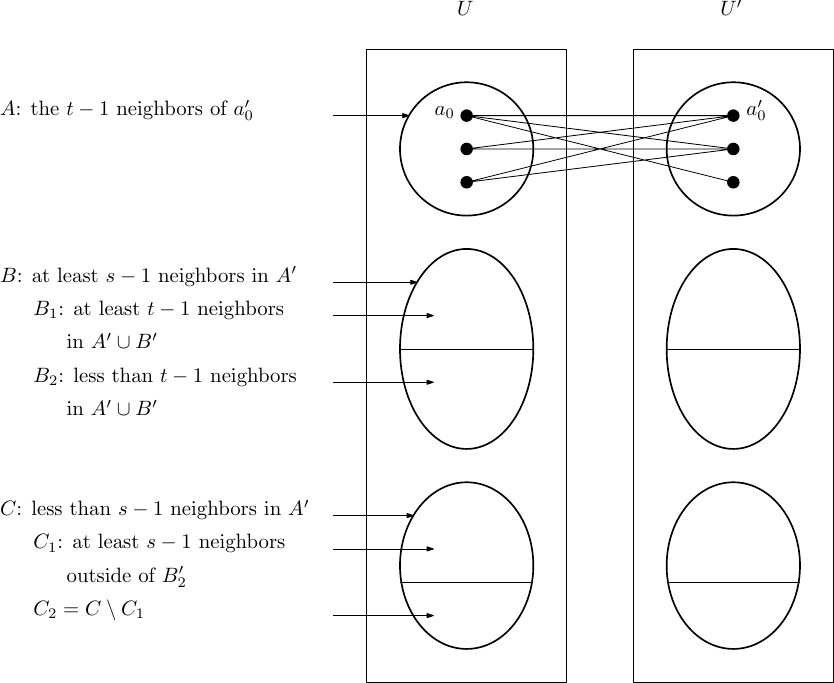}
    \caption{Example of a nice core and its related partitions}
    \label{fig:core}
\end{figure}

\begin{observation} \label{obs}
    Every vertex $v\in C$ has at least $t-1$ neighbors in $A'\cup B'$. Similarly, every vertex in $C'$ has at least $t-1$ neighbors in $A\cup B$. 
\end{observation}

\begin{proof}
    Let $v\in C$. Consider adding the edge $va_0'$ and the created copy of $K_{s,t}$. Since $a_0'$ is part of the $K_{s,t}$, the only vertices in $U$ other than $v$ that can be used are those in $A$. Since $v$ has less than $s-1$ neighbors in $A'$, the $K_{s,t}$ uses at least one vertex outside of $A'$. This implies $a_0$ is not part of this $K_{s,t}$ and thus at most $t-2$ vertices are used from $A$. Then we can only create a copy of $K_{(s, t)}$. Since $s-1$ vertices in $A$ are used, it follows that the copy of $K_{(s, t)}$ does not contain any vertices in $C'$. Thus, $v$ has at least $t-1$ neighbors in $A'\cup B'$. By symmetry, the mirrored statement is also true. 
\end{proof}

\begin{observation} \label{obs1}
    The edges between $C_2$ and $C_2'$ form a complete bipartite graph. 
\end{observation}

\begin{proof}
    Suppose for the sake of contradiction that $vv'$ is not an edge, where $v\in C_2$ and $v'\in C_2'$. Consider adding the edge $vv'$. Without loss of generality, assume that it creates a copy of $K_{(s, t)}$. Since $v'\in C_2'$ has at most $s-2$ neighbors outside of $B_2$, the $K_{(s, t)}$ uses at least one vertex $u\in B_2$. With similar reasoning on $v$, the $K_{(s, t)}$ uses at least $t-s+1$ vertices in $B_2'$. This implies $u$ has at least $t-s$ neighbors in $B_2'$. Since  $u$ also has at least $s-1$ neighbors in $A'$, it has at least $t-1$ neighbors in $A'\cup B'$, a contradiction. 
\end{proof}

    Let us first concentrate on the case when $s = t-1$. By definition, this implies that 1) the core $(A, A')$ forms a complete bipartite graph, 2) every vertex in $B$ and $B'$ is incident to the vertices in $A'\setminus \{a_0'\}$ and $A\setminus \{a_0\}$ respectively and 3) every vertex in $B_2$ and $B_2'$ has no neighbors in $B'$ and $B$ respectively. We will prove the following claim, which in turn helps us to get the required lower bound on the number of edges of $G$.

    \begin{claim} \label{three}
    One of the following holds about $G$:
    \begin{enumerate}
        \item Either $B_2$ or $B_2'$ is empty.
        \item Either $C_2$ or $C_2'$ is empty.
        \item Either all the vertices in $B_2$ have at least $t-1$ neighbors outside of $C_2'$ or the mirror opposite occurs.
    \end{enumerate}
    \end{claim}
    
    We comment here that since the first two cases both vacuously imply the third one, the last case always holds. However, for the convenience of analysis, we still state the first two cases separately. 
    
    \begin{proof}
        For the sake of contradiction, let us assume that none of $B_2$, $B_2'$, $C_2$, and $C_2'$ are empty, and there are two vertices $v \in B_2$ and $v' \in B_2'$ such that $v$ and $v'$ have less than $t-1$ neighbors outside of $C_2'$ and $C_2$ respectively. Since $s=t-1$, by definition, the vertices in $B_2$ and $B_2'$ do not have any neighbors in $B'$ and $B$ respectively. Then, there are no edges between $B_2$ and $B_2'$. If we add the edge $vv'$, then the created copy of $K_{t-1,t}$ must use at least one vertex from $C_2$ or $C_2'$ because of the degree restriction of $v$ and $v'$ outside of $C_2'$ and $C_2$ respectively. Without loss of generality, say it creates a copy of $K_{(t, t-1)}$ and uses a vertex $u\in C_2$. We know that $u$ has less than $t-2$ neighbors outside of $B_2'$, so there should be at least two vertices from $B_2'$ participating in the created copy of $K_{(t,t-1)}$. Hence, there is a vertex $w' \in B_2'$ other than $v'$ inside the $K_{(t,t-1)}$. This implies that $vw'$ must be an edge of $G$, which contradicts the fact that there are no edges between $B_2$ and $B_2'$. 
    \end{proof}
    
    In the next couple of claims we prove that all three cases of Claim \ref{three} would force $G$ to be either one of the graphs in $\F_{s,t}^n$ or a strictly sub-optimal graph.
 
    \begin{claim} \label{case2}
    If either $B_2$ or $B_2'$ is empty, then $G \in \F_{t-1,t}^n$. 
    \end{claim}
    
    \begin{proof}
        Without loss of generality, assume that $B_2$ is empty. If $C'$ is empty, then every vertex in $U'$ is adjacent to the vertices in $A\setminus a_0$, resulting in $G$ being a graph in $\F_{t-1,t}^n$. Thus, we can assume that $C'$ is non-empty. Now by Observation \ref{obs}, the number of edges in $G$ is the following:
        
        \begin{align*}
            e(U, U') &= e(A, A') + e(B_1, A' \cup B') + e(C, A' \cup B') + e(B', A) + e(C', U) \\
            &\ge (t-1)^2 + (t-1)|B_1| + (t-1)|C| + (t-2)|B'| + (t-1)|C'| \\
            &> (2t-3)n - (t-1)(t-2),
        \end{align*}
        where the last strict inequality follows from the fact that $|C'| > 0$. This strict inequality contradicts the fact that $G$ has the minimum number of edges among all $K_{t-1,t}$-saturated $n$ by $n$ bipartite graphs.
    \end{proof}
    
    \begin{claim} \label{case3}
    If either $C_2$ or $C_2'$ is empty, then $G \in \F_{t-1,t}^n$.
    \end{claim}
    
    \begin{proof}
        Without loss of generality, assume that $C_2'$ is empty. By Observation \ref{obs}, the number of edges in $G$ is the following:
        
        \begin{align}
            e(U, U') &= e(A, A') + e(B_1, A' \cup B') + e(B_2, U') + e(C, A' \cup B') + e(B', A) + e(C_1', U \setminus B_2) \nonumber \\
            &\ge (t-1)^2 + (t-1)|B_1| + (t-1)|B_2| + (t-1)|C| + (t-2)|B'| + (t-2)|C_1'| \label{eq1} \\
            &= (2t-3)n - (t-1)(t-2). \nonumber
        \end{align}
        
        We now wish to show that if $G$ achieves equality, then $G$ must be one of the graphs in $\F_{t-1,t}^n$. Note that if $C_1'$ is also empty, then we already have a graph from $\F_{t-1,t}^n$. Then, we can assume that $C_1'$ is non-empty. Using Claim \ref{case2}, we can assume that $B_2$ and $B_2'$ both are non-empty. Pick $v \in B_2$ and $v' \in B_2'$ (by (3) after Observation \ref{obs1}, $vv'$ is not an edge in $G$). After adding the edge $vv'$, we split into two cases depending on whether a copy of $K_{(t-1,t)}$ or a copy of $K_{(t,t-1)}$ is created.
        \smallskip
        
        \textbf{Case 1}: adding $vv'$ creates a copy of $K_{(t-1,t)}$. Since $v$ has exactly $t-2$ neighbors in $A'$ and no neighbors in $B'$, there exists $w'\in C_1'$ that is part of the $K_{(t-1, t)}$. In this copy among the $t-1$ vertices in $U$, there can be at most $t-3$ vertices from $A$ (because of the presence of $w'$) and there cannot be any vertices in $B$ other than $v$ (because of the presence of $v'$). Then, there must be $w \in C$ participating in the $K_{(t-1,t)}$. Now, among the $t$ vertices in $U'$, there can be at most $t-3$ vertices from $A$ (due of the presence of $w$) and there cannot be any vertices in $B'$ other than $v'$. This implies there are at least two vertices from $C_1'$ taking part in the $K_{(t-1,t)}$, proving that $v \in B_2$ has at least $t$ neighbors in $U'$. Then, we have that $e(B_2, U') > (t-1)|B_2|$, which in turn makes the inequality in \eqref{eq1} a strict one. Hence, $G$ would have a sub-optimal number of edges, a contradiction.
        \smallskip
        
        \textbf{Case 2}: adding $vv'$ creates a copy of $K_{(t,t-1)}$. Due to similar reasoning as the previous case, there must be a $w \in C$ participating in the $K_{(t,t-1)}$. Again with similar reasons as before, there is also $w' \in C_1'$ taking part in the $K_{(t,t-1)}$. Now, note that this $K_{(t,t-1)}$ uses $t-1$ vertices in $U \setminus B_2$. Then, $w'$ has at least $t-1$ neighbors outside of $B_2$, proving that $e(C_1', U \setminus B_2) > (t-2)|C_1'|$. This again shows that the inequality in \eqref{eq1} is a strict one, which again leads to a contradiction like last time.    
    \end{proof}

   \begin{claim} \label{case1}
    If all the vertices in $B_2$ or $B_2'$ have at least $t-1$ neighbors outside of $C_2'$ or $C_2$ respectively and $C_2$ and $C_2'$ are non-empty, then $G$ has strictly sub-optimal edges, which is a contradiction.
    \end{claim}
    
    \begin{proof}
        Without loss of generality, assume all the vertices in $B_2$ have at least $t-1$ neighbors outside of $C_2'$. In this case, $e(B_1, A' \cup B') + e(B_2, U' \setminus C_2') \ge (t-1)|B_1| + (t-1)|B_2| = (t-1)|B|$, giving us the following inequalities:
        
        \begin{align*}
            e(U, U') &= e(A, A') + e(B_1, A' \cup B') + e(B_2, U' \setminus C_2') + e(C, A' \cup B') + e(B', A) \\
            &\;\;\;\;\;\;\;\;\;\;\;\;\;\;\;\;\;\;\;\;\;\;\;\;\;\;\;\;\;\;\;\;\;\;\;\;\;\;\;\;\;\;\;\;\;\;\;\;\;\;\;\;\;\;\;\;\;\;\;\;\;\;\;\;\;\;\;\;\;\;\;\;\;\;\; + e(C_1', U \setminus B_2) + e(C_2', U) \\
            &\ge (t-1)^2 + (t-1)|B| + (t-1)|C| + (t-2)|B'| + (t-2)|C_1'| + (t-1)|C_2'| \\
            &> (2t-3)n - (t-1)(t-2),
        \end{align*}
        where the last strict inequality follows from the fact that $|C_2'| > 0$. This results in a contradiction.
    \end{proof}

    Note that by combining all previous claims, we have effectively characterized the extremal graphs for $s=t-1$, proving the moreover part of Lemma \ref{nicecore}. To prove the rest of the lemma, we now consider the case of general $s$ and use the following claim from \cite{GKS}.
    
    \begin{claim}[Claim 2.3 in \cite{GKS}] \label{core:bound}
    Assuming $\d \ge t-1$, and the number of edges in the core is $e$, we have that $$e(U, U') \ge (s+t-2)(n-t+1) - \left\lfloor \frac{(s-1)^2}{4}\right\rfloor + e.$$
    \end{claim}
    
    We repeat the proof of this claim from \cite{GKS} for the sake of completeness. 
    
    \begin{proof}
        Let $y = |C_2|$ and $y' = |C_2'|$. Without loss of generality, assume that $y \le y'$. It follows from the discussion after the statement of Lemma \ref{nicecore} and Observation \ref{obs} that $e(B_1', A \cup B) \ge (t-1)|B_1'|$, $e(B_2', U) \ge (t-1)|B_2'|$, and $e(C', A \cup B) \ge (t-1)|C'|$. Then, we have the following:
        
        \begin{align}
            e(B_1', A \cup B) + e(B_2', U) + e(C', A \cup B) &\ge (t-1)|B_1'| + (t-1)|B_2'| + (t-1)|C'| \nonumber \\
            &= (t-1)(n-t+1). \label{prime}
        \end{align}
        
        Now we count the rest of the edges. Notice that $e(B, A') \ge (s-1)|B|$, $e(C_1, A' \cup B_1' \cup C') \ge (s-1)|C_1|$, and $e(C_2, C_2') = yy' \ge y^2$ by Observation \ref{obs1}. Then, we have the following:
        
        \begin{align}
            e(B, A') + e(C_1, A' \cup B_1' \cup C') + e(C_2, C_2') &\ge (s-1)|B| + (s-1)|C_1| + y^2 \nonumber \\
            &= (s-1)(n-t+1-y) + y^2 \nonumber \\
            &\ge (s-1)(n-t+1) - \left\lfloor \frac{(s-1)^2}{4}\right\rfloor. \label{nonprime}
        \end{align}
        
        Hence, Claim \ref{core:bound} follows from the Equations \eqref{prime} and \eqref{nonprime}.
    \end{proof}

\begin{proof}[Proof of Lemma \ref{nicecore}]
    As mentioned before, the moreover part follows immediately from Claims \ref{three} - \ref{case1}. Now, note that in the nice core, each of $a_0$ and $a_0'$ has degree $t-1$ and the core contains a copy of $K_{s,t-1}$. This implies there are at least $s(t-1)+(t-1-s)$ edges in the nice core. Then, Lemma \ref{nicecore} follows immediately from Claim \ref{core:bound}. 
\end{proof}

\section{$K_{2,4}$-saturation}
\label{sec:smallcase}

As discussed in the introduction, it is a simple exercise to determine $\sat\left(K_{n,n}, K_{1,t}\right)$. It is also easy to see that $\sat\left(K_{n,n}, K_{t,t}\right)$ is the same as the `ordered' version $\sat\left(K_{n,n}, K_{(t,t)}\right)$, which was already completely solved. In this paper, we have determined the value of $\sat\left(K_{n,n}, K_{t-1,t}\right)$ for all sufficiently large $n$. Thus, the next smallest open case is to find $\sat\left(K_{n,n}, K_{2,4}\right)$. Next, we show that Conjecture \ref{conjecture} holds for the case of $K_{2,4}$-saturation.

\begin{proposition} \label{k24}
For all sufficiently large $n$, $\sat\left(K_{n,n}, K_{2,4}\right) = 4n - 4$. 
\end{proposition}

\begin{proof}
    The examples stated after Conjecture \ref{conjecture} already establish the upper bound of $4n-4$. For the lower bound, consider a graph $G$ which is a $K_{2,4}$-saturated $n$ by $n$ bipartite graph with minimum number of edges. As in our proof of the main results, we split into two cases depending on the minimum degree of $G$. In both cases, we mirror the strategy of Proposition \ref{mindeg}. Let $\d$ denote the minimum degree of $G$.
    \smallskip

    \textbf{Case 1}: $\delta \le 2$. As $G$ is $K_{2,4}$-saturated, we have that $\d \ge 1$. Similar to the definitions in the proof of Proposition \ref{mindeg}, let $u_0$ be a vertex of degree $\delta$. Without loss of generality, assume $u_0\in U$. For every vertex $u'\in U'\setminus N(u_0)$, adding an edge $u_0u'$ must create a copy of $K_{(4, 2)}$ and thus let $S_{u'}$ be the set of the three neighbors of $u'$ participating in that copy. Note that every vertex in $S_{u'}$ is adjacent to a vertex in $N(u_0)$. Let $V$ denote $\cup_{u'\in U'\setminus N(u_0)}S_{u'}$. Then, performing similar counts as in the proof of Proposition \ref{mindeg}, we have the following:
    
    \begin{equation} \label{mineq}
        e(U,U') \ge |V| + 3(n-\d) + \d(n-|V|).
    \end{equation}
        
    If $\d = 1$, then Equation \eqref{mineq} implies that $e(U,U') \ge 4n-3$, proving Proposition \ref{k24}. Then, the only remaining case to check is when $\d = 2$. In this case, since $a_0\notin V$, we have that $n-|V|\ge 1$ and Equation \eqref{mineq} becomes: 
        
    \begin{equation} \label{equality}
    e(U,U') \ge |V| + 3(n-2) + 2(n-|V|) \ge 4n-5.
    \end{equation}
        
    Thus it remains to show that equality cannot hold in Equation \eqref{equality}. Assume for the sake of contradiction that equality holds. By the proof of Proposition \ref{mindeg}, we can conclude that if equality is satisfied in Equation \eqref{equality}, then we must have $|V| = n-1$; in other words $V = U \setminus \{u_0\}$. Moreover, every vertex in $V$ must be adjacent to exactly one vertex in $N(u_0)$. Furthermore, if we assume that $N(u_0) = \{u_1', u_2'\}$, then every vertex in $V' = U' \setminus N(u_0)$ (clearly, $|V'|=n-2$) must be adjacent to exactly three vertices in $V$. More importantly, this counting accounts for all the edges. Then, by the Handshake Lemma, there exists a vertex $v \in V$ such that $v$ has at most two neighbors in $V'$. Since $v$ has exactly one neighbor in $N(u_0)$, without loss of generality, assume that $vu_1'$ is not an edge.
        
    Consider adding the edge $vu_1'$. Note that it cannot create a copy of $K_{(2,4)}$. Otherwise, the copy must use $v$, some other vertex $w \in U$, $u_1'$, and all three neighbors of $v$ ($v$ must have degree exactly three). However, since every vertex in $V$ has only one neighbor in $\{u_1', u_2'\}$, the vertex $u_0$ is the only candidate for $w$ but $u_0$ is not adjacent to any vertex in $V'$, contradicting the existence of the copy of $K_{(2, 4)}$. 
        
    Thus, adding $vu_1'$ creates a copy of $K_{(4,2)}$. In particular, it contains a vertex in $U'$ of degree at least $4$. Since all vertices in $V'$ have degree exactly three, the $K_{(4, 2)}$ must use $u_2'$. However, there do not exist three vertices in $U$ that are adjacent to both $u_1'$ and $u_2'$, a contradiction.
    \smallskip

    \textbf{Case 2}: $\delta\ge 3$. Then, Lemma \ref{findcore} implies that $G$ contains a nice core. In particular, the core has three vertices on each side, contains $a_0\in U$ and $a_0'\in U'$ that are adjacent to the three vertices in the core on the other side, and the core contains a copy of $K_{2, 3}$. This implies the core has at least seven edges and from Claim \ref{core:bound}, we have that the number of edges in $G$ is at least $4(n-3) + 7 = 4n-5$. Thus, our goal is to show that the equality does not hold.
        
    Assume for the sake of contradiction that equality holds. It follows from Claim \ref{core:bound} that the nice core has exactly seven edges and thus must be a $K_{3, 3}$ with two edges removed. Since the core contains a copy of $K_{2, 3}$, the removed edges are incident to the same vertex. Without loss of generality, let $a_1'\in U'$ be that vertex. Note that $a_0$ is the only neighbor of $a_1'$ inside the core. 
        
    Next we show that for $u' \in U' \setminus A'$, adding any edge $a_0u'$ to $G$ would create a copy of $K_{(4,2)}$. Suppose not, then a copy of $K_{(2,4)}$ is created and all the vertices in $A'$ are used in that copy. However, it can be observed that $a_0'$ and $a_1'$ do not have another common neighbor other than $a_0$, contradicting the creation of a $K_{(2, 4)}$. Hence, adding $a_0u'$ creates a copy of $K_{(4,2)}$. Now observe that the vertex $a_0$ behaves similarly to the vertex $u_0$ in the proof of Proposition \ref{mindeg}. Thus, we again define a couple of sets in the style of Section 2 to help us count the number of edges of $G$.
        
    For any $u' \in U' \setminus A'$, let $S_{u'}$ denote the set of three vertices in $U \setminus a_0$ participating in a created $K_{(4,2)}$ when the edge $a_0u'$ is added. Define $V \subseteq U$ to be the union of these $S_{u'}$'s. Clearly, any vertex in $V$ must be adjacent to at least one neighbor in $A'$, and all the vertices in $U' \setminus A'$ must be adjacent to at least three vertices in $V$. Now, we can count the number of edges in $G$ similarly to the proof of Proposition \ref{mindeg}:
    
    \begin{align}
        4n - 5 &=e(U, U') \nonumber \\
        &= e(A, A') + e(V \setminus A, A') + e(V, U' \setminus A') + e(U \setminus (V \cup A), U') + e(A \setminus V, U' \setminus A') \nonumber \\
        &\ge 7 + |V \setminus A| + 3(n-3) + 3(n - 3 - |V \setminus A|) + 0 \label{rv1} \\
        &\ge 4n - 5. \label{rv2}
    \end{align}
        
    Then, all the steps are equal to $4n-5$. Since the last inequality, in Step \eqref{rv2}, is an equality, we have that $|V \setminus A| = n-3$ (i.e., $V \cup A = U$). Similarly, from the inequality in Step \eqref{rv1}, we can conclude that every vertex in $V\setminus A$ has exactly one neighbor in $A'$, every vertex in $U'\setminus A'$ has exactly three neighbors in $V$, and there are no edges between $A\setminus V$ and $U'\setminus A'$. 
        
    First, we show that in fact $V = U \setminus \{a_0\}$. Consider $a\in A$ that is not $a_0$. Since the minimum degree is at least three and $a$ has only two neighbors in $A'$, there is an edge between $a$ and $U' \setminus A'$. Since $e(A \setminus V, U' \setminus A') = 0$, it follows that $a\in V$. Thus, we have established that $V$ contains all the vertices in $U$ except for $a_0$.
        
    Thus, it follows from the inequality used in Step \eqref{rv1} that every vertex in $U \setminus A$ has exactly one neighbor in $A'$ and every vertex in $U' \setminus A'$ has exactly three neighbors in $U$. Now, since the minimum degree of $G$ is at least three, the vertex $a_1'$ has a neighbor $u$ in $U \setminus A$. Note that $u$ is only incident to $a_1'$ in $A'$.
        
    Now, consider adding the edge $ua_0'$. We claim it cannot create a copy of $K_{(4,2)}$. Suppose it does, then all the vertices in $A$ need to participate in the copy (because they are the only three neighbors of $a_0'$). However, the vertices in $A$ do not share a common neighbor with $u$, contradicting the existence of this $K_{(4, 2)}$. 
        
    Then adding $ua_0'$ creates a copy of $K_{(2,4)}$ which uses a vertex $u' \in U' \setminus A'$ and a vertex $a \in A \setminus \{a_0\}$. Recall that all vertices in $U'\setminus A'$ have degree exactly three, and $a$ and $u$ are neighbors of $u'$. As the final step, consider adding the edge $a_0u'$ and show it cannot create a copy of $K_{2, 4}$. 
        
    If adding $a_0u'$ creates a copy of $K_{(2,4)}$, then that copy must use all the vertices in $A'$. However, $a_0'$ and $a_1'$ do not have another common neighbor other than $a_0$, a contradiction. If adding $a_0u'$ creates a copy of $K_{(4,2)}$, then it uses $a_0$ along with all three neighbors of $u'$. However, $a_0$, $a$, and $u$ do not have any common neighbors in $U'$ (recall $a$ and $u$ are neighbors of $u'$), a contradiction. This contradiction completes the proof of Proposition \ref{k24}.
\end{proof}


\section{Concluding remarks}
\label{sec:conclu}

In this paper, we split the proof of Theorem \ref{main1} into three big cases, and they were discussed in Sections 2, 3, and 4. Note that the arguments in Section 3 work for general $s$ and $t$ to find the nice core as defined in Section 2. It remains to be seen if one can improve the structural arguments in Sections 2 and 4 for general $s$ and $t$. Another interesting problem in graph saturation is to impose degree restrictions on $G$ (see, e.g., \cite{AEHK}, \cite{FS}, and \cite{P}). For example, one can ask if the $K_{s,t}$-saturated ($s < t$) extremal graphs with minimum degree at least $t-1$ are also $K_{(s,t)}$ or $K_{(t,s)}$-saturated graph. Perhaps one can modify the proof of Lemma \ref{nicecore} for general $s$ instead of just $s=t-1$ to answer such questions. A different direction one can also explore is to investigate Conjecture \ref{conjecture} for $s=2$, i.e., the $K_{2,t}$ case. 

Note that our lower bounds in Theorems \ref{main1} and \ref{main2} also work even if we consider another related notion called strong saturation. A graph $G$ is called strongly $H$-saturated if adding any missing edge to $G$ creates a new copy of $H$, but there is no requirement of $G$ being $H$-free to begin with. Similar to the definition of $\sat$, we define $\ssat\left(K_{n,n}, K_{s,t}\right)$ to be the minimum number of edges among all strongly $K_{s,t}$-saturated $n$ by $n$ bipartite graphs. Interestingly, in our arguments of finding the lower bounds in Theorem \ref{main1} and Theorem \ref{main2}, we never use the fact that $G$ is $K_{s,t}$-free, thus the same lower bounds still hold for strong saturation. Similar to Theorem \ref{main1}, we can classify the extremal graphs for strong saturation in the case of $s = t-1$.  

Recently, there is a trend of studying generalized graph saturation problems (which are defined analogously to the generalized extremal problems, see e.g., \cite{AS}). The generalized saturation problem asks to minimize the number of copies of a fixed graph $T$ instead of minimizing the number of edges.  Curious readers can look at \cite{CL} and \cite{KMTT} for some extensions of the classical results in graph saturation. Perhaps, one can consider the generalized saturation problems in the bipartite setting, such as studying the minimum number of copies of $K_{a,b}$ in a $K_{s,t}$-saturated bipartite graph. This can be investigated in the `ordered' or `weak saturation' settings as well.

\section{Acknowledgement}
We are grateful to the anonymous referees for helping us improve the quality of the presentation of this paper.

\end{document}